\documentclass[reqno]{amsart}

\usepackage{draftwatermark}
\SetWatermarkLightness{0.9}
\SetWatermarkText{---~PREPRINT~---~\today~---}
\SetWatermarkScale{2}

\usepackage{amsmath} 
\usepackage{amssymb} 
\usepackage{amsthm} 
\usepackage{empheq}
\usepackage{graphicx}
\usepackage[colorlinks,linkcolor=blue]{hyperref} 
\usepackage[loose]{subfigure}
\usepackage{psfrag}

\newcommand{\rd}{\mathrm{d}}
\newcommand{\D}{\mathrm{D}}
\newcommand{\edp}{\widetilde{\Psi}}
\newcommand{\Elas}{\mathcal{E}}
\newcommand{\E}{\mathbb{E}}
\newcommand{\ep}{\varepsilon}
\newcommand{\R}{\mathbb{R}}
\newcommand{\Stab}{\mathcal{S}}

\newtheorem{theorem}{Theorem}[section]
\newtheorem{proposition}[theorem]{Proposition}
\newtheorem{lemma}[theorem]{Lemma}
\theoremstyle{definition}
\newtheorem{definition}[theorem]{Definition}

\makeatletter
\@addtoreset{equation}{section}
\@addtoreset{figure}{section}
\@addtoreset{table}{section}
\makeatother

\hypersetup{
	pdftitle = {Thermalization of rate-independent processes by entropic regularization},
	pdfauthor = {T. J. Sullivan, M. Koslowski, F. Theil, M. Ortiz},
	pdfsubject = {2010 MSC: 47J35, 82C35},
	pdfkeywords = {gradient descent, non-linear evolution equations, thermodynamics},
}

\title[Thermalization of rate-independent processes]{Thermalization of rate-independent processes by entropic regularization}

\author{T.\ J.\ Sullivan}
\address{
	T.\ J.\ Sullivan \\
	Applied \& Computational Mathematics and Graduate Aerospace Laboratories \\
	California Institute of Technology \\
	Mail Code 9-94 \\
	1200 East California Boulevard \\
	Pasadena \\
	CA 91125-9400 \\
	USA
}
\email{tjs@caltech.edu}
\urladdr{\url{http://www.its.caltech.edu/~tjs/}}

\author{M. Koslowski}
\address{
	M. Koslowski \\
	School of Mechanical Engineering \\
	Purdue University \\
	585 Purdue Mall \\
	West Lafayette \\
	IN 47907-2088 \\
	USA
}
\email{marisol@purdue.edu}

\author{F.\ Theil}
\address{
	F.\ Theil \\
	Mathematics Institute \\
	University of Warwick \\
	Coventry \\
	CV4 7AL \\
	UK
}
\email{f.theil@warwick.ac.uk}
\urladdr{\url{http://www.warwick.ac.uk/~masfk/}}

\author{M.\ Ortiz}
\address{
	M.\ Ortiz \\
	Graduate Aerospace Laboratories \\
	California Institute of Technology \\
	Mail Code 105-50 \\
	1200 East California Boulevard \\
	Pasadena \\
	CA 91125 \\
	USA
}
\email{ortiz@caltech.edu}
\urladdr{\url{http://www.aero.caltech.edu/~ortiz/}}

\date{\today}

\thanks{The first author acknowledges portions of this work supported by the Engineering and Physical Sciences Research Council (\url{http://www.epsrc.ac.uk/}).}



\begin{document}

\begin{abstract}
	We consider the effective behaviour of a rate-independent process when it is placed in contact with a heat bath.  The method used to ``thermalize'' the process is an interior-point entropic regularization of the Moreau--Yosida incremental formulation of the unperturbed process.  It is shown that the heat bath destroys the rate independence in a controlled and deterministic way, and that the effective dynamics are those of a non-linear gradient descent in the original energetic potential with respect to a different and non-trivial effective dissipation potential.
\end{abstract}

\maketitle

\section{Introduction and Outline}

In \cite{Koslowski:2003, SullivanKoslowskiTheilOrtiz:2009}, it was proposed that a suitable model for the effect of a heat bath (\emph{i.e.}\ the application of statistically disordered energy) on a gradient descent is a time-incremental variational problem in which, in each time step, the usual work done competes with an entropy term that penalizes coherent, deterministic evolutions.  In the case of linear kinetics (two-homogeneous dissipation), this method is equivalent to the one used in \cite{JordanKinderlehrerOtto:1998} to generate the Fokker--Planck equation for an It{\=o} stochastic gradient descent.  This paper examines the case of one-homogeneous dissipation and generalizes the results of \cite{SullivanKoslowskiTheilOrtiz:2009}.

As outlined in Section \ref{sec:notation}, the discrete-time formulation of a rate-independent evolution in an energetic potential $E(t, x)$ with respect to a one-homogeneous dissipation potential $\Psi(x)$ is to find, given state $x_{i}$ at time $t_{i}$, the state $x_{i + 1}$ at time $t_{i + 1}$ that minimizes
\begin{equation}
	\label{eq:intro_work}
	W_{i + 1}(x_{i}, x_{i + 1}) := E(t_{i + 1}, x_{i + 1}) - E(t_{i}, x_{i}) + \Psi(x_{i + 1} - x_{i}).
\end{equation}
To represent the influence of a heat bath of ``intensity'' $\theta > 0$ upon this evolution, we consider an associated variational problem \eqref{eq:IP-Cost-2} for the probability distribution of the \emph{random} next state of the system, the solution of which is the Gibbsian density
\begin{equation}
	\label{eq:intro_Gibbs}
	\rho_{i + 1}(x_{i + 1}|x_{i}) \propto \exp \left( - \frac{W_{i + 1}(x_{i}, x_{i + 1})}{\theta (t_{i + 1} - t_{i})} \right).
\end{equation}
This paper shows that, under suitable assumptions on $E$ and $\Psi$, in the limit as the time step tends to zero, this procedure yields a non-trivial \emph{deterministic} limiting process.  This limiting process is a gradient descent in the original energetic potential $E$ but with respect to a new dissipation potential $\edp$ that is a non-linear transformation (the Cramer transform) of the original one $\Psi$.  As demonstrated in \cite{Sullivan:2009, SullivanKoslowskiTheilOrtiz:2009}, this non-linear gradient descent arises in mechanical contexts such as Andrade creep.

Rate-independent processes play an important r{\^o}le in the modelling of many physical phenomena such as plasticity and phase transformations in elastic solids, electromagnetism, dry friction on surfaces, and pinning problems in superconductivity.  It is widely accepted that rate-independent processes, which describe mesoscopic or macroscopic properties, are limit processes for more complicated microstructural evolutions:  the rate-independent model arises in the limit of vanishing inertia, relaxation time and thermal effects.  Hence, this paper is concerned with the relaxation of the third of these limiting assumptions.

In Section \ref{sec:notation}, the notation and set-up of the problem are given, including a brief review of the necessary elements of the theories of gradient descents and rate-independent processes.  In Section \ref{sec:heuristics}, some formal calculations are performed that motivate the introduction of the effective dissipation potential $\edp$.  In Section \ref{sec:main}, $\edp$ is defined more formally, its properties examined, and the main convergence theorem (Theorem \ref{thm:main}) is stated.  Some conclusions and outlook for future work are given in Section \ref{sec:conclusions}.  The proofs of the various results are given in Section \ref{sec:proofs}.

\section{Notation and Set-Up of the Problem}
\label{sec:notation}

\subsection{Gradient Descents}

Both the unperturbed and perturbed processes of study in this paper are examples of gradient descents.  The standard example of a gradient descent is the ordinary differential equation $\dot{x}(t) = - \nabla E(t, x(t))$ for $x \colon [0, T] \to \R^{n}$, which is characterized by the energy evolution law
\begin{equation}
	\label{eq:E_diff_classical}
	\frac{\rd}{\rd t} E(t, x(t)) = (\partial_{t} E)(t, x(t)) - \frac{1}{2} | \dot{x}(t) |^{2} - \frac{1}{2} | \nabla E(t, x(t)) |^{2}.
\end{equation}
In general, gradient descents may be considered on any metric space $(\mathcal{Q}, d)$;  see \cite{AmbrosioGigliSavare:2008} for a comprehensive treatment.  For the purposes of this paper, however, it is enough to consider the case in which $\mathcal{Q}$ is a subset of a Banach space $(\mathcal{X}, \| \cdot \|)$.

A gradient descent in $\mathcal{Q}$ is characterized by an initial condition $x_{0} \in \mathcal{Q}$, an energetic potential $E \colon [0, T] \times \mathcal{Q} \to \R$, and a dissipation potential $\Psi \colon \mathcal{X} \to [0, +\infty]$, which is convex and satisfies $\Psi(0) = 0$.  For simplicity, $E(t, x)$ is assumed to be differentiable with respect to both $t$ and $x$.

\begin{definition}
	\label{defn:MetGD}
	An absolutely continuous curve $x \colon [0, T] \to \mathcal{Q}$ is said to be a \emph{gradient descent} in $E$ with respect to $\Psi$ and starting at $x_{0}$ if
	\begin{enumerate}
		\item $x(0) = x(0+) = x_{0}$;
		\item $t \mapsto E(t, x(t))$ is absolutely continuous;
		\item the (differential) energy inequality
		\begin{equation}
			\label{eq:E_diff_Banach}
			\frac{\rd}{\rd t} E(t, x(t)) \leq (\partial_{t} E)(t, x(t)) - \Psi ( \dot{x}(t) ) - \Psi^{\star} ( - \D E(t, x(t)) )
		\end{equation}
		is satisfied for almost every $t \in [0, T]$, where $\Psi^{\star} \colon \mathcal{X}^{\ast} \to [0, +\infty]$ denotes the convex conjugate (Legendre--Fenchel transform) of $\Psi$, defined by
		\begin{equation}
			\label{eq:convex_conjugate}
			\Psi^{\star}(\ell) := \sup \{ \langle \ell, x \rangle - \Psi(x) \mid x \in \mathcal{X} \}.
		\end{equation}
	\end{enumerate}
\end{definition}

The condition \eqref{eq:E_diff_Banach} is the appropriate generalization of \eqref{eq:E_diff_classical};  the classical case of linear kinetics is that in which the dissipation potential is given by $\Psi(x) := \frac{1}{2} \| x \|^{2}$.  Shortly, we shall consider rate-independent processes, in which $\Psi$ is positively homogeneous of degree one;  the limiting processes of this paper will be gradient descents for which $\Psi$ is not homogeneous of any degree.

\subsection{Incremental Formulation}

The analysis and numerical approximation of gradient descents are often performed using a discrete-time incremental variational formulation.  At each time step, the problem is to minimize the Moreau--Yosida regularization of $E(t_{i}, \cdot)$ \cite{Moreau:1965, Yosida:1965}.  $P$ will denote a partition of the interval of time $[0, T]$, \emph{i.e.}\ a finite strictly increasing sequence
\begin{equation}
	\label{eq:partition}
	P = \{ 0 = t_{0} < t_{1} < \dots < t_{N} = T \},
\end{equation}
where $\Delta t_{i} := t_{i} - t_{i - 1}$ and $[P]$ denotes the \emph{mesh size} of $P$:
\begin{equation}
	\label{eq:mesh}
	[P] := \max_{i = 1, \dots, N} | \Delta t_{i} |.
\end{equation}

The Moreau--Yosida scheme is a causal sequence of variational problems, the Euler--Lagrange equations of which are the equations of motion for the original gradient descent:

\begin{definition}
	\label{defn:MoreauYosida}
	The \emph{Moreau--Yosida incremental formulation} of the gradient descent in $E$ with respect to $\Psi$ is to solve the following sequence of minimization problems:  given an initial condition $x_{0}^{(P)} = x_{0} \in \mathcal{Q}$, find, for $i = 0, \ldots, N - 1$, $x_{i + 1}^{(P)} \in \mathcal{Q}$ to minimize
	\begin{equation}
		\label{eq:MoreauYosida}
		E \big( t_{i + 1}, x_{i + 1}^{(P)} \big) - E \big( t_{i}, x_{i}^{(P)} \big) + \Delta t_{i + 1} \Psi \left( \frac{\Delta x_{i + 1}^{(P)}}{\Delta t_{i + 1}} \right).
	\end{equation}
	By abuse of notation, let $x^{(P)} \colon [0, T] \to \mathcal{Q}$ also denote the c\`{a}dl\`{a}g piecewise-constant interpolation of the sequence $\big( x_{i}^{(P)} \big)_{i = 0}^{N}$, as defined by
	\begin{equation}
		\label{eq:cadlag_interpolation}
		x^{(P)}(t) := x_{i}^{(P)} \text{ for } t \in [t_{i}, t_{i + 1}).
	\end{equation}
\end{definition}

\subsection{Rate-Independent Processes}

A rate-independent process is an evolutionary system that has no intrinsic time-scale:  it ``reacts only as quickly as its time-dependent inputs''.  Put another way, the solution operator commutes with monotone reparametrizations of time.  There is much literature on the theory, modelling and analysis of rate-independent processes and the connections with gradient descent theory; see \emph{e.g.}\ \cite{Mielke:2005, Mielke:2007, MielkeRossiSavare:2009}.

\begin{definition}
	\label{defn:rate-independence}
	Let $\mathcal{Q}$ and $\mathcal{Q}^{\ast}$ be topological spaces.  Suppose that each choice of initial condition $x_{0} \in \mathcal{Q}$ and each input $\ell \colon [t_{0}, t_{1}] \to \mathcal{Q}^{\ast}$ determines a set of outputs
	\[
		\mathcal{O}([t_{0}, t_{1}], x_{0}, \ell) \subseteq \{ x \colon [t_{0}, t_{1}] \to \mathcal{Q} \mid x(t_{0}) = x_{0} \}.
	\]
	The input-output relationship is said to be \emph{rate-independent} if, for every strictly increasing and surjective $\varphi \colon [t'_{0}, t'_{1}] \to [t_{0}, t_{1}]$,
	\[
		x \in \mathcal{O}([t_{0}, t_{1}], x_{0}, \ell) \iff x \circ \varphi \in \mathcal{O}([t'_{0}, t'_{1}], x_{0}, \ell \circ \varphi).
	\]
	The relationship is said to determine a (possibly multi-valued) \emph{evolutionary system} if concatenations and restrictions of solutions are also solutions, \emph{i.e.}
	\begin{align*}
		& \hat{x} \in \mathcal{O}([t_{0}, t_{1}], x_{0}, \ell|_{[t_{0}, t_{1}]}), \tilde{x} \in \mathcal{O}([t_{1}, t_{2}], x_{1}, \ell|_{[t_{1}, t_{2}]}) \text{ and } \hat{x}(t_{1}) = x_{1} \\
		& \implies x \in \mathcal{O}([t_{0}, t_{2}], x_{0}, \ell) \text{ where } x(t) := \begin{cases} \hat{x}(t), & \text{if $t \in [t_{0}, t_{1}]$,} \\ \tilde{x}(t), & \text{if $t \in [t_{1}, t_{2}]$;} \end{cases}
	\end{align*}
	and
	\begin{align*}
		& x \in \mathcal{O}([t_{0}, t_{1}], x_{0}, \ell), [s_{0}, s_{1}] \subseteq [t_{0}, t_{1}] \text{ and } y_{0} := x(s_{0}) \\
		& \implies x|_{[s_{0}, s_{1}]} \in \mathcal{O}([s_{0}, s_{1}], y_{0}, \ell|_{[s_{0}, s_{1}]}).
	\end{align*}
\end{definition}

In the case of gradient descents on (subsets of) Banach spaces as described above, rate-independence corresponds to the dissipation potential $\Psi \colon \mathcal{X} \to [0, +\infty]$ being positively homogeneous of degree one, \emph{i.e.}
\begin{equation}
	\label{eq:pos_homo_deg_1}
	\Psi(\alpha x) = \alpha \Psi(x) \text{ for all $\alpha > 0$, $x \in \mathcal{X}$.}
\end{equation}
It will be assumed that $\Psi$ is both continuous and non-degenerate:  \emph{i.e.}\ there exist constants $c_{\Psi}, C_{\Psi} > 0$ such that
\begin{equation}
	\label{eq:Psi_cC}
	c_{\Psi} \| x \| \leq \Psi(x) \leq C_{\Psi} \| x \| \text{ for all } x \in \mathcal{X}.
\end{equation}
This is equivalent to assuming that $\Psi$ is the convex conjugate of the characteristic function of a suitable subset of $\mathcal{X}^{\ast}$:
\begin{equation}
	\label{eq:DissElas}
	\Psi(x) = \chi_{\Elas}^{\star} (x) = \sup \left\{ \langle \ell, x \rangle \,\middle|\, \ell \in \Elas \right\}
\end{equation}
for some bounded, closed and convex set $\Elas \subseteq \mathcal{X}^{\ast}$ having $0$ as an interior point.  $\Elas$ is known as the \emph{elastic region} and its frontier $\partial \Elas$ is known as the \emph{yield surface}.  The set
\begin{equation}
	\label{eq:stable}
	\Stab(t) := \{ x \in \mathcal{Q} \mid - \D E(t, x) \in \Elas \}
\end{equation}
is the collection of (locally) \emph{stable states} at time $t$;  since in this paper the energy $E$ will always be convex, the distinction between global and local stability will not matter.

As shown in \cite[Theorem 7.1]{MielkeTheil:2004}, the rate-independent problem is well-posed in the case that $\mathcal{Q} = \mathcal{X}$ is a separable and reflexive Banach space;  that $\Psi$ satisfies \eqref{eq:pos_homo_deg_1} and \eqref{eq:Psi_cC};  and that $E(t, \cdot)$ is of smoothness class $\mathcal{C}^{3}$, with the eigenvalues of $\D^{2} E$ bounded below by some $\gamma > 0$, uniformly in time and space.

\subsection{Thermalized Gradient Descents: Entropic Regularization}

Consider a gradient descent in $\R^{n}$ with respect to an energy $E$ and dissipation $\Psi$.  The corresponding Moreau--Yosida incremental problem is as follows:  given the state $x_{i}$ at time $t_{i}$, the aim is to find $x_{i + 1}$ to minimize
\[
	W_{i + 1}(x_{i}, x_{i + 1}) := E(t_{i + 1}, x_{i + 1}) - E(t_{i}, x_{i}) + \Delta t_{i + 1} \Psi \left( \frac{\Delta x_{i + 1}}{\Delta t_{i + 1}} \right).
\]
To model the effect of a heat bath on the gradient descent, we pass to an extended problem, in which the state of the system at time $t_{i}$ is a random variable $X_{i}$.  

Given that the random state $X_{i}$ assumes the value $x_{i}$ at time $t_{i}$, the random next state $X_{i + 1}$ for time $t_{i + 1}$ is posited to have the conditional probability density function $\rho_{i + 1}(\cdot | x_{i}) \in L^{1}(\R^{n}, \lambda; [0, +\infty])$ that minimizes
\begin{equation}
	\label{eq:IP-Cost-2}
	\int_{\R^{n}} \big( W_{i + 1}(x_{i}, \cdot) \rho_{i + 1} (\cdot|x_{i}) + \ep_{i + 1} \rho_{i + 1} (\cdot|x_{i}) \log \rho_{i + 1} (\cdot|x_{i}) \big) \, \rd \lambda,
\end{equation}
where $\lambda$ denotes Lebesgue measure.  The parameter $\ep_{i + 1} > 0$ represents the intensity of the heat bath to which the gradient descent is coupled;  more precisely, $\ep_{i + 1}$ is the amount of (disordered) energy that the heat bath injects into the system over the time interval $[t_{i}, t_{i + 1}]$.  

Equivalently to \eqref{eq:IP-Cost-2}, given that the current state $X_{i}$ has probability density function $\rho_{i} \in L^{1}(\R^{n}, \lambda; [0, +\infty])$, we may seek a joint probability density function $\rho_{i, i + 1} \in L^{1}(\R^{n}  \times \R^{n}, \lambda \otimes \lambda; [0, +\infty])$ that has $\rho_{i}$ as its first marginal and minimizes
\begin{equation}
	\label{eq:IP-Cost}
	\iint_{\R^{n}} \big( W_{i + 1} \rho_{i, i + 1} + \ep_{i + 1} \rho_{i, i + 1} \log \rho_{i, i + 1} \big) \, \rd (\lambda \otimes \lambda).
\end{equation}
The connection between \eqref{eq:IP-Cost-2} and \eqref{eq:IP-Cost} is given by
\[
	\rho_{i, i + 1}(x_{i}, x_{i + 1}) = \rho_{i}(x_{i}) \rho_{i + 1} (x_{i + 1}|x_{i}).
\]

The minimizer of \eqref{eq:IP-Cost-2} is a Gibbs--Boltzmann-type conditional probability density function (\emph{cf.}\ \cite{JordanKinderlehrer:1996, JordanKinderlehrerOtto:1998}):
\begin{equation}
	\label{eq:Gibbsian_pdf}
	\rho_{i + 1} (x_{i + 1}|x_{i}) = \frac{\exp \left( - W_{i + 1}(x_{i},x_{i + 1}) / \ep_{i + 1} \right)}{\int_{\R^{n}} \exp \left( - W_{i + 1}(x_{i},x_{i + 1}) / \ep_{i + 1} \right) \rd x_{i + 1}},
\end{equation}

Hence, given a partition $P$ of $[0, T]$, an initial state $x_{0} \in \R^{n}$, an energetic potential $E \colon [0, T] \times \R^{n} \to \R$ and a dissipation potential $\Psi \colon \R^{n} \to [0, +\infty)$, the \emph{thermalized gradient descent} $X^{(P)}$ denotes the discrete-time Markov chain that has transition probability densities given by \eqref{eq:Gibbsian_pdf}.  By the usual abuse of notation, $X^{(P)}$ will also denote the c{\`a}dl{\`a}g piecewise-constant interpolation \eqref{eq:cadlag_interpolation}, defined for all times $t \in [0, T]$.

In the classical case of linear kinetics (\emph{i.e.}\ $\Psi(x) = \frac{1}{2} | x |^{2}$ for $x \in \R^{n}$), this procedure generates the same sequence of densities as the method of \cite{JordanKinderlehrerOtto:1998}, and they converge as $[P] \to 0$ to the solution of the Fokker--Planck equation for the It{\=o} stochastic gradient descent $\dot{X}(t) = - \nabla E(t, X(t)) + \sqrt{\ep} \dot{W}(t)$.  Theorem \ref{thm:main} establishes the deterministic limiting behaviour of the stochastic process $X^{(P)}$ as $[P] \to 0$ in the case of a one-homogeneous dissipation potential $\Psi$.

\section{Heuristics and Calculation of Moments}
\label{sec:heuristics}

In this section we perform some calculations to motivate the main result of Section \ref{sec:main}.  For simplicity, suppose temporarily that $E$ is of the prototypical quadratic type
\[
	E(t, x) = \tfrac{1}{2} \langle A x, x \rangle - \langle \ell(t), x \rangle
\]
for some symmetric and non-negative $A \colon \R^{n} \to (\R^{n})^{\ast}$ and some smooth enough $\ell \colon [0, T] \to (\R^{n})^{\ast}$;  this assumption will be relaxed shortly.  Also, merely to aid the heuristic and simplify the notation, suppose that the parameters $\ep_{i} > 0$ are all equal to some constant $\ep > 0$ independent of $i$ and that the time step $\Delta t_{i}$ is also independent of $i$.
 
Consider the following calculation for the conditional expectation of the next state $X_{i + 1}^{(P)}$ of the Markov chain $X^{(P)}$ given that $X_{i}^{(P)} = x_{i}$:
\begin{align*}
	& \E \left[ X_{i + 1}^{(P)} \,\middle|\, X_{i}^{(P)} = x_{i} \right] \\
	& \quad = \int_{\R^{n}} x_{i + 1} \rho_{i + 1} ( x_{i + 1} \mid x_{i} ) \, \rd x_{i + 1} \\
	& \quad = \frac{\displaystyle \int_{\R^{n}} x_{i + 1} \exp \big( - (E(t_{i + 1}, x_{i + 1}) - E(t_{i}, x_{i}) + \Psi(x_{i + 1} - x_{i}))/ \ep \big) \, \rd x_{i + 1}}{\displaystyle \int_{\R^{n}} \exp \big( - (E(t_{i + 1}, x_{i + 1}) - E(t_{i}, x_{i}) + \Psi(x_{i + 1} - x_{i}))/ \ep \big) \, \rd x_{i + 1}} \\
	& \quad = \frac{\displaystyle \int_{\R^{n}} x_{i + 1} \exp \big( - (E(t_{i + 1}, x_{i + 1}) + \Psi(x_{i + 1} - x_{i}))/ \ep \big) \, \rd x_{i + 1}}{\displaystyle \int_{\R^{n}} \exp \big( - (E(t_{i + 1}, x_{i + 1}) + \Psi(x_{i + 1} - x_{i}))/ \ep \big) \, \rd x_{i + 1}} \\
	\intertext{and setting $z := (x_{i + 1} - x_{i}) / \ep$ yields}
	& \quad = x_{i} + \ep \frac{\displaystyle \int_{\R^{n}} z \exp \big( - ( \langle A x_{i} - \ell(t_{i + 1}), z \rangle + \tfrac{\ep}{2} \langle A z, z \rangle + \Psi(z)) \big) \, \rd z}{\displaystyle \int_{\R^{n}} \exp \big( - ( \langle A x_{i} - \ell(t_{i + 1}), z \rangle + \tfrac{\ep}{2} \langle A z, z \rangle + \Psi(z)) \big) \, \rd z}.
\end{align*}
Let
\begin{equation}
	\label{eq:Psitilde_eps_quadratic}
	\edp_{\ep}^{\star} (w) := \log \int_{\R^{n}} \exp \big( - (\langle w, z \rangle + \tfrac{\ep}{2} \langle A z, z \rangle + \Psi(z)) \big) \, \rd z.
\end{equation}
Then the result of the above calculation may be summarized as
\[
	\E \left[ \Delta X_{i + 1}^{(P)} \,\middle|\, X_{i}^{(P)} = x_{i} \right] = \left. - \ep \D \edp_{\ep}^{\star} (w) \right|_{w = A x_{i} - \ell(t_{i + 1})},
\]
\emph{i.e.}
\[
	\E \left[ \Delta X_{i + 1}^{(P)} \,\middle|\, X_{i}^{(P)} = x_{i} \right] = - \ep \D \edp_{\ep}^{\star} (\D E(t_{i + 1}, x_{i})).
\]

The same change of variables $z := (x_{i + 1} - x_{i}) / \ep$ gives an estimate for the $p^{\mathrm{th}}$ moment of the increments of the Markov chain:
\begin{align*}
	& \E \left[ \big| \Delta X_{i + 1}^{(P)} |^{p} \,\middle|\, X_{i}^{(P)} = x_{i} \right] \\
	& \quad \leq \ep^{p} \frac{\displaystyle \int_{\R^{n}} | z |^{p} \exp \big( - ( \langle A x_{i} - \ell(t_{i + 1}), z \rangle + \tfrac{\ep}{2} \langle A z, z \rangle + \Psi(z)) \big) \, \rd z}{\displaystyle \int_{\R^{n}} \exp \big( - ( \langle A x_{i} - \ell(t_{i + 1}), z \rangle + \tfrac{\ep}{2} \langle A z, z \rangle + \Psi(z)) \big) \, \rd z},
\end{align*}

For later reference, these calculations are summarized in the following lemma:

\begin{lemma}
	\label{lem:Markov-ExpAndMoments}
	Let $E(t, x) = \frac{1}{2} \langle A x, x \rangle - \langle \ell(t), x \rangle$ with $A \colon \R^{n} \to (\R^{n})^{\ast}$ symmetric and non-negative.  Suppose also that $\Psi = \chi_{\Elas}^{\star} \colon \R^{n} \to [0, + \infty)$ is $1$-homogeneous and non-degenerate.  Let $X^{(P)}$ denote the thermalized gradient descent Markov chain in $E$ and $\Psi$ on a partition $P$ of $[0, T]$.  Then
	\[
		\E \left[ \Delta X_{i + 1}^{(P)} \,\middle|\, X_{i}^{(P)} = x_{i} \right] = - \ep \D \edp_{\ep}^{\star} (A x_{i} - \ell(t_{i + 1})).
	\]
	and, for $p > 0$,
	\begin{align*}
		& \E \left[ \big| \Delta X_{i + 1}^{(P)} |^{p} \,\middle|\, X_{i}^{(P)} = x_{i} \right] \\
		& \quad \leq \ep^{p} \frac{\displaystyle \int_{\R^{n}} | z |^{p} \exp \big( - ( \langle A x_{i} - \ell(t_{i + 1}), z \rangle + \tfrac{\ep}{2} \langle A z, z \rangle + \Psi(z)) \big) \, \rd z}{\displaystyle \int_{\R^{n}} \exp \big( - ( \langle A x_{i} - \ell(t_{i + 1}), z \rangle + \tfrac{\ep}{2} \langle A z, z \rangle + \Psi(z)) \big) \, \rd z}.
	\end{align*}
\end{lemma}

The above calculations, including Lemma \ref{lem:Markov-ExpAndMoments}, also go through even if $E$ is not a quadratic form.  The non-$\Psi$ terms in the exponent are the Taylor series expansion of $E(t_{i + 1}, x_{i + 1}) - E(t_{i + 1}, x_{i})$ about $x_{i}$ and, therefore, the corresponding expression for $\edp_{\ep}^{\star}$ is
\begin{equation}
	\label{eq:Psitilde_eps}
	\edp_{\ep}^{\star} (w) := \log \int_{\R^{n}} \exp \left( - \left( \langle w, z \rangle + \sum_{k = 2}^{\infty} \tfrac{\ep^{k - 1}}{k !} \langle \D^{k} E(t_{i + 1}, x_{i}), z^{\otimes k} \rangle + \Psi(z) \right) \right) \rd z.
\end{equation}
By abuse of notation, Lemma \ref{lem:Markov-ExpAndMoments} will henceforth be taken to refer to the generalized result for not-necessarily-quadratic $E$ using \eqref{eq:Psitilde_eps}.

Note, however, that in none of these expressions does the time increment appear explicitly.  This is to be expected, since the original evolution was a rate-independent one.  Therefore, in order to obtain a Markov chain that takes any account of time, it will be necessary to take $\ep$ to be proportional to the time step.  Physically, since $E$, $\Psi$ and $\ep$ all have the units of energy, this corresponds to assuming that the heat bath supplies energy to the system at a constant rate:  the power of the heat bath is the constant of proportionality $\theta$ between $\ep$ and the time step.  The parameter $\theta$ measures the intensity of the heat bath and can be seen, in some sense, as the ``temperature''.

The potential $\edp_{\ep}^{\star} \colon (\R^{n})^{\ast} \to [0, + \infty]$ encodes a great deal of information about the Markov chain $X$.  Most of the terms in the exponent of $\edp_{\ep}^{\star}$ are of order $\ep$ or higher, and so can reasonably be expected to have no influence in the limit as $[P]$ tends to zero in proportion to $\ep$.  The limiting dynamics of the Markov chain are expected to be controlled by an \emph{effective dual dissipation potential} $\edp^{\star}$, which is $\edp_{\ep}^{\star}$ with these higher-order terms omitted.  Furthermore, the strong similarity to the Euler method for an ordinary differential equation and the fact that the variances are of order $\ep^{2} \ll \ep$ suggest that the limiting evolution takes the form of a deterministic ordinary differential equation
\begin{equation}
	\label{eq:limiting_ode}
	\dot{y}(t) = - \theta \D \edp^{\star} (\D E(t, y(t))),
\end{equation}
where $\theta = \ep_{i} / \Delta t_{i}$.  By convex duality, \eqref{eq:limiting_ode} is equivalent to
\begin{equation}
	\label{eq:limiting_gd1}
	\D \edp \left( - \frac{\dot{y}(t)}{\theta} \right) = \D E(t, y(t)).
\end{equation}
If $\Psi$ is even (\emph{i.e.}\ $\Psi(x) = \Psi(-x)$), then so is $\edp$, in which case \eqref{eq:limiting_gd1} is equivalent to the non-linear gradient descent
\begin{equation}
	\label{eq:limiting_gd2}
	\D \edp \left( \frac{\dot{y}(t)}{\theta} \right) = - \D E(t, y(t)).
\end{equation}
Therefore, the conjecture is that the effective behaviour of the rate-independent process in $E$ with respect to $\Psi$ when brought into contact with the heat bath is that of a gradient descent in $E$ with respect to the non-linear effective dissipation potential $\edp$.

\section{Main Results}
\label{sec:main}

In this section the formal manipulations of the previous section are made more precise:  the effective (dual) dissipation potential that corresponds to $\Psi$ is introduced and its properties examined;  and the main convergence theorem about the limiting behaviour of the thermalized gradient descent Markov chain $X^{(P)}$ as $[P] \to 0$ is stated.

\subsection{Effective Dissipation Potential}

As mentioned above, the effective dual dissipation potential $\edp^{\star}$ is simply the functional $\edp_{\ep}^{\star}$ of \eqref{eq:Psitilde_eps} with $\ep$ set equal to zero, and $\edp$ is its convex conjugate:

\begin{definition}
	Given $\Psi \colon \R^{n} \to [0, +\infty]$ homogeneous of degree one, define the associated \emph{effective dual dissipation potential} $\edp^{\star} \colon (\R^{n})^{\ast} \to [0, + \infty]$ by
	\begin{equation}
		\label{eq:EDDP}
		\edp^{\star} (w) := \log \int_{\R^{n}} \exp ( - (\langle w, z \rangle + \Psi(z)) ) \, \rd z.
	\end{equation}
	The associated \emph{effective dissipation potential} $\edp \colon \R^{n} \to [0, + \infty]$ is the Cramer transform of $\Psi$ and is defined by convex conjugation:  $\edp := (\edp^{\star})^{\star}$, \emph{i.e.}
	\begin{equation}
		\label{eq:EDP}
		\edp (x) := \sup \left\{ \langle w, x \rangle - \edp^{\star} (w) \,\middle|\, w \in (\R^{n})^{\ast} \right\}.
	\end{equation}
\end{definition}

Up to a minus sign, $\edp^{\star}$ is the logarithmic moment generating function (or cumulant generating function) of the Borel measure $\psi$ on $\R^{n}$ defined by
\begin{equation}
	\label{eq:psi_measure}
	\rd \psi(z) := \exp (- \Psi(z)) \, \rd z.
\end{equation}

It is often convenient to write $\edp^{\star}$ as an integral over the Euclidean unit sphere $\mathbb{S}^{n - 1} \subsetneq \R^{n}$ with respect to $(n - 1)$-dimensional Hausdorff measure $\mathcal{H}^{n - 1}$:
\begin{equation}
	\label{eq:EDDP-Spherical}
	\edp^{\star} (w) = \log \int_{\mathbb{S}^{n - 1}} \frac{(n - 1)!}{( \langle w, \omega \rangle + \Psi(\omega) )^{n}} \, \rd \mathcal{H}^{n - 1} (\omega).
\end{equation}

Note that $\edp$ and $\edp^{\star}$ are objects that are intrinsic to the dissipation, not the energetic structure:  they are determined entirely by the duality between $\R^{n}$ and $(\R^{n})^{\ast}$ and the dissipation potential $\Psi$ (or, equivalently, the geometry of the elastic region $\Elas$).  Proposition \ref{prop:EDDP} summarizes the important properties of the effective dual dissipation potential $\edp^{\star}$;  the proof is deferred to Section \ref{sec:proofs}.

\begin{figure}[t]
	\begin{center}
		\subfigure[$\edp^{\star}(w) = - \log ( 1 - w^{2} )$.]{
			\psfrag{w}{{\small $w$}}
			\psfrag{F}[c]{{\small $\edp^{\star}(w)$}}
			\includegraphics[width=0.45\textwidth]{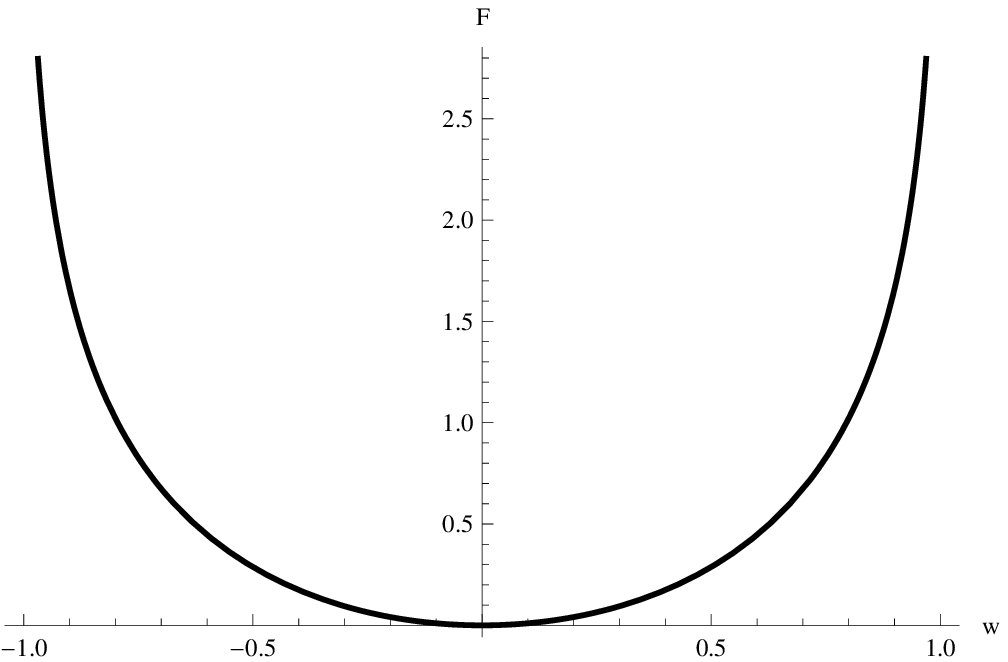}
		}
		\subfigure[$\Psi$ (dashed), and $\edp$ (solid).]{
			\psfrag{x}{{\small $x$}}
			\psfrag{Fs}[c]{{\small $\Psi(x), \edp(x)$}}
			\includegraphics[width=0.45\textwidth]{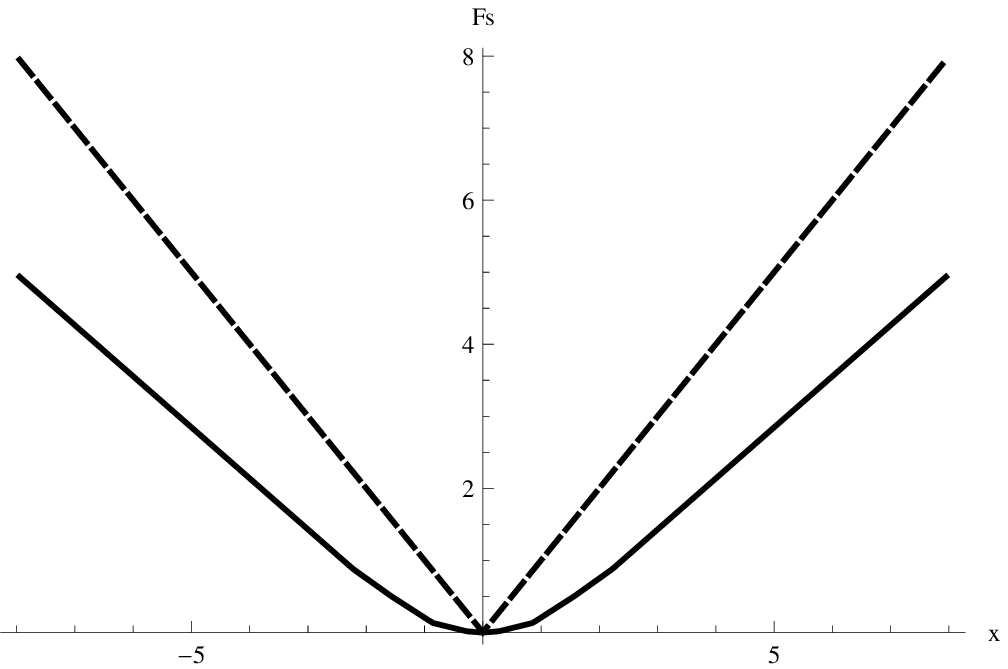}
		}
	  \caption{The effective (dual) dissipation potential in dimension one, with dissipation potential $\Psi(x) := | x |$.  Note the linear growth of $\edp$ for large $| x |$ and its approximate $2$-homogeneity near the origin.}
	  \label{fig:EDP}
	\end{center}
\end{figure}

\begin{proposition}
	\label{prop:EDDP}
	If $\Psi = \chi_{\Elas}^{\star} \colon \R^{n} \to \R$ satisfies \eqref{eq:Psi_cC}, then $\edp^{\star} \colon (\R^{n})^{\ast} \to [0, + \infty]$ defined as in \eqref{eq:EDDP} satisfies
	\begin{enumerate}
		\item $\edp^{\star}(w) > 0$ for all $w \in (\R^{n})^{\ast}$;
		\item $\edp^{\star}(w) < + \infty \iff -w \in \mathring{\Elas}$;
		\item $\edp^{\star}$ is convex on $(\R^{n})^{\ast}$;
		\item $\edp^{\star}$ is smooth on $- \mathring{\Elas}$;
		\item $\edp^{\star}(w)$ and $| \D \edp^{\star}(w) | \to + \infty$ as $-w \to \partial \Elas$.
	\end{enumerate}
\end{proposition}

Proposition \ref{prop:EDDP} immediately implies that $\edp$ is smooth and strictly convex.  In some special cases of interest, $\edp^{\star}$ can be determined explicitly:
\begin{enumerate}
	\item Suppose that the elastic region $\Elas$ is a rectangular box with faces perpendicular to the coordinate axes in $(\R^{n})^{\ast}$:
	\begin{equation}
		\label{eq:elas_Manhattan}
		\Elas := \big\{ w = (w_{1}, \ldots, w_{n}) \in (\R^{n})^{\ast} \big| | w_{i} | \leq \sigma_{i} \text{ for } i = 1, \ldots, n \big\}.
	\end{equation}
	Then the dissipation potential $\Psi$ is the weighted $\ell^{1}$ ``Manhattan'' norm $\Psi(z) = \sum_{j = 1}^{n} \sigma_{j} | z^{j} |$ and
	\begin{equation}
		\label{eq:EDDP_Manhattan}
		\edp^{\star} (w) = - \sum_{i = 1}^{n} \log \big( \sigma_{i}^{2} - | w_{i} |^{2} \big).
	\end{equation}
	\item Suppose that the elastic region $\Elas$ is a Euclidean ball
	\begin{equation}
		\label{eq:elas_Euclidean}
		\Elas := \big\{ w = (w_{1}, \ldots, w_{n}) \in (\R^{n})^{\ast} \big| | w_{1} |^{2} + \dots + | w_{n} |^{2} \leq \sigma^{2} \big\}.
	\end{equation}
	Then the dissipation potential $\Psi$ is exactly $\sigma$ times the usual Euclidean norm and
	\begin{equation}
		\label{eq:EDDP_Euclidean}
		\edp^{\star} (w) = - \frac{n + 1}{2} \log \big( \sigma^{2} - | w |^{2} \big).
	\end{equation}
\end{enumerate}

\subsection{Convergence Theorem}

To the standing assumption that $\Psi = \chi_{\Elas}^{\star}$ satisfies \eqref{eq:pos_homo_deg_1} and \eqref{eq:Psi_cC}, we now add some assumptions on the energetic potential $E$.  $E \colon [0, T] \times \R^{n} \to \R$ is assumed to be bounded below, smooth in space with all derivatives uniformly bounded, and such that $(t, x) \mapsto \D E(t, x)$ is uniformly Lipschitz.  It is also assumed that $E$ is convex, and hence that the Hessian of $E$ is a non-negative operator.  Two further, more technical, assumptions are also required.  Both of these assumptions are satisfied in the prototypical case
\[
	E(t, x) := \tfrac{1}{2} \langle A x, x \rangle - \langle \ell(t), x \rangle,
\]
where $\ell \colon [0, T] \to (\R^{n})^{\ast}$ is Lipschitz, and $A \colon \R^{n} \to (\R^{n})^{\ast}$ is symmetric and non-negative.  In this case, the stable region at time $t \in [0, T]$ is the preimage
\[
	\Stab(t) = A^{-1} ( \ell(t) - \Elas )
\]
and is convex and closed for every $t$;  if $A$ is positive-definite, then $\Stab(t)$ is also bounded, and hence compact.  The prototypical case was examined in \cite{SullivanKoslowskiTheilOrtiz:2009};  the technical conditions that follow were introduced in \cite{Sullivan:2009}.

In order to control certain error terms in the proof of Lemma \ref{lem:ParabolicTightnessDiscrete}, which leads to Theorem \ref{thm:main}, a monotonicity assumption is used to ensure that these terms have the right sign regardless of their magnitude.  The requisite assumption is that
\begin{equation}
	\label{eq:NastyConvexityAssumption}
	\text{for all } t \in [0, T], x \mapsto \edp^{\star} (\D E(t, x)) \text{ is convex,}
\end{equation}
or, equivalently, that $\D \edp^{\star} (\D E(t, \cdot))$ is a monotone vector field for every $t \in [0, T]$.  This is a non-trivial assumption even if $E$ is strictly convex, as the example illustrated in Figure \ref{fig:NonConvexEDDP} shows.  Note also that \eqref{eq:NastyConvexityAssumption} presupposes that the set $\Stab(t)$ of stable states is convex for every $t \in [0, T]$, and that convexity of $E(t, \cdot)$ does not imply convexity of $\mathcal{S}(t)$ --- see \emph{e.g.}\ the kidney-shaped stable set of \cite[Example 5.5]{MielkeTheil:2004}.  Nevertheless, \eqref{eq:NastyConvexityAssumption} holds in the prototypical case, since $\D E(t, x) = A x - \ell(t)$ is an affine function and the composition of convex function with an affine one always yields a convex function \cite[\S3.2]{BoydVandenberghe:2004}.

\begin{figure}[t]
	\begin{center}
	  \subfigure[Comparison of $V$ (convex, dashed) and $\edp^{\star} \circ \D V$ (non-convex, solid).]
    {
			\psfrag{x}{{\small $x$}}
			\psfrag{e}[c]{{\small $V(x), \edp^{\star}(\D V(x))$}}
			\includegraphics[width=0.45\textwidth]{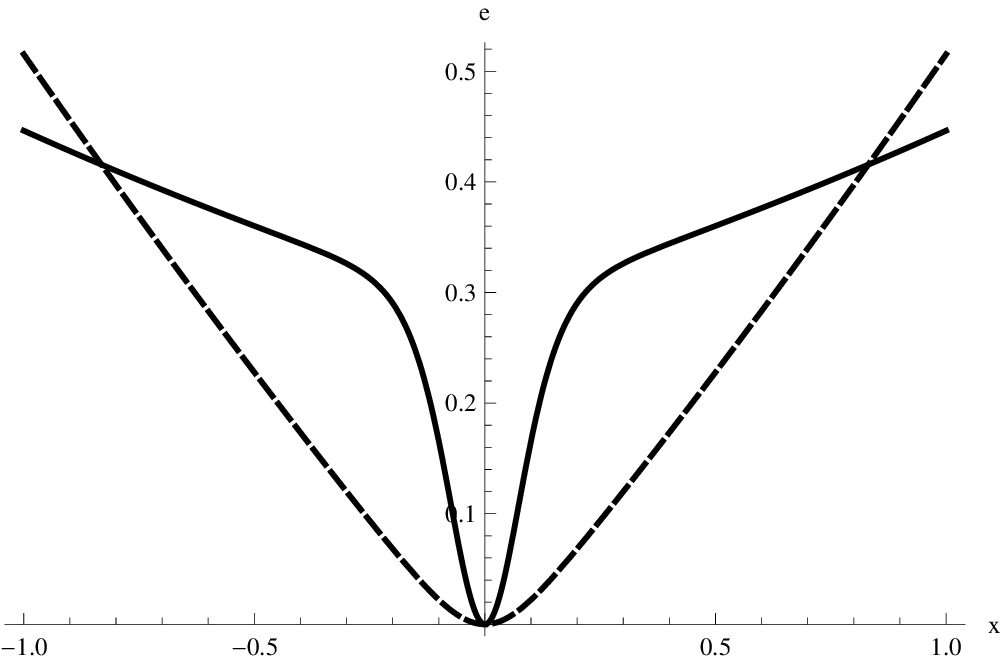}
    }
	  \subfigure[The potential gradient $\D V$.]
    {
			\psfrag{x}{{\small $x$}}
			\psfrag{de}[c]{{\small $\D V(x)$}}
			\includegraphics[width=0.45\textwidth]{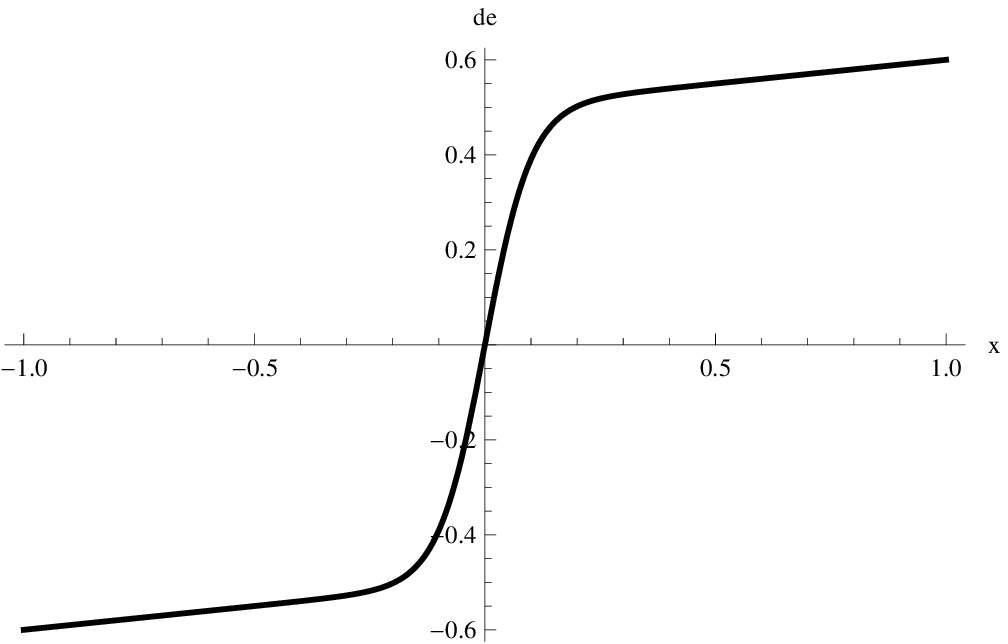}
    }
	  \caption[A convex potential $V$ for which $\edp^{\star} \circ \D V$ is not convex.]{An example of a strictly convex potential $V$ for which $\edp^{\star} \circ \D V$ is not convex.  In dimension $1$, consider $\Psi(x) := | x |$ and ${ V(x) := \tfrac{x^{2}}{20} + \tfrac{1}{20} \log \cosh ( 10 x ) }$.  In this case, $\edp^{\star}(w) = - \log (1 - w^{2})$ and $\D V(x) = \tfrac{1}{2} \tanh (10 x) + \tfrac{1}{10} x$. The composition $x \mapsto \edp^{\star}(\D V(x))$ is evidently non-convex, although it is quasiconvex (\emph{i.e.}\ it has convex sublevel sets).}
	  \label{fig:NonConvexEDDP}
	\end{center}
\end{figure}

It is also necessary to place an implicit constraint on the time-dependency of $E$.  The problem to be avoided is that all the estimates for the moments of the increments $\Delta X_{i + 1}^{(P)}$ blow up as $- \D E(t_{i}, X_{i}^{(P)})$ approaches the yield surface $\partial \Elas$.  The situation to be avoided can be expressed neatly in terms of the proposed limiting deterministic process and the effective dual dissipation potential.  Therefore, we impose the following finite energy criterion:
\begin{equation}
	\label{eq:MarkovLim-Convex-USC}
	\left.
	\begin{matrix}
		\mathcal{T} \text{ an interval of time starting at } 0, \\
		y(0) \text{ such that } - \D E(0, y(0)) \in \mathring{\Elas}, \\
		\dot{y} = - \theta \D \edp^{\star} (\D E(t, y)) \text{ on } \mathcal{T}
	\end{matrix}
	\right\}
	\implies \sup_{t \in \mathcal{T}} \edp^{\star} (\D E(t, y(t))) < + \infty.
\end{equation}

Condition \eqref{eq:MarkovLim-Convex-USC} is somewhat implicit, but appears to be unavoidable if energies that have neither identically zero nor constant positive-definite Hessian are to be considered.  If $E$ is of the prototypical quadratic form with $A = 0$, then \eqref{eq:MarkovLim-Convex-USC} holds if the applied load $\ell$ is never greater than the dissipation, \emph{i.e.}\ if
\[
	\inf_{t \in [0, T]} \inf_{| z | = 1} \left( \Psi(z) - \ell(t) \cdot z \right) > 0.
\]
(\emph{N.B.}  In this case of ``small $\ell$'', the rate-independent process is static but the thermalized process \eqref{eq:limiting_gd1} is not.)  If $A$ is positive-definite, then \eqref{eq:MarkovLim-Convex-USC} always holds whenever the initial condition satisfies $y(0) \in \mathring{\Stab}(0) = A^{-1} ( \ell(0) - \Elas )$.  Indeed, for not-necessarily-quadratic energies, uniform convexity of $E$ implies the condition \eqref{eq:MarkovLim-Convex-USC}:

\begin{lemma}
	\label{lem:unif_cvx_implies_good}
	Suppose that there exists $\gamma_{E} > 0$ such that
	\[
		\langle \D^{2} E(t, x) , v \otimes v \rangle \geq \gamma_{E} | v |^{2} \text{ for all $t \in [0, T)$ and all $x, v \in \R^{n}$,}
	\]
	and that $\| \partial_{t} \D E \|_{L^{\infty}} < + \infty$.  Then \eqref{eq:MarkovLim-Convex-USC} holds.
\end{lemma}

See Section \ref{sec:proofs} for the proof of Lemma \ref{lem:unif_cvx_implies_good}.  The main result is now as follows:

\begin{theorem}
	\label{thm:main} %
	Suppose that $E$, $\Psi$ satisfy the hypotheses above, including \eqref{eq:NastyConvexityAssumption} and \eqref{eq:MarkovLim-Convex-USC}, and fix $\theta > 0$.  Then, as $[P] \to 0$, the piecewise constant c{\`a}dl{\`a}g interpolants of $X^{(P)}$ with $\ep_{i} = \theta \Delta t_{i}$ converge in probability in the uniform norm to the deterministic non-linear gradient descent $y \colon [0, T] \to \R^{n}$ satisfying \eqref{eq:limiting_gd1}, with the same initial condition $X_{0} = y(0) = x_{0} \in \mathring{\Stab}(0)$.  More precisely, for any $T > 0$, $\eta > 0$, there exists a constant $C \geq 0$ such that, for all small enough $[P]$,
	\begin{equation}
		\label{eq:MarkovLim-Convex-Estimate}
		\mathbb{P} \left[ \sup_{t \in [0, T]} \big| X^{(P)}(t) - y(t) \big| \geq \eta \right] \leq C [P]^{1/2}.
	\end{equation}
\end{theorem}

\begin{proof}
	The claim follows from the standard $O([P])$ global error bound for deterministic Euler schemes, the $O(\ep^{1/2})$ estimate of Lemma \ref{lem:FepsToF0Flows}, and the $O([P])$ estimate of Lemma \ref{lem:ParabolicTightnessDiscrete}.
\end{proof}

An illustrative comparison of the original rate-independent evolution and the effect of the heat bath is given in Figure \ref{fig:RI-Relaxed}.  Note that when $\theta$ is large (which corresponds to the heat bath being very hot), $\dot{y}(t) / \theta$ typically lies in the region of $\R^{n}$ close to the origin where $\edp$ is approximately $2$-homogeneous;  when $\theta$ is small (which corresponds to the heat bath being cold), $\dot{y}(t) / \theta$ typically lies in the region of $\R^{n}$ far from the origin where $\edp$ is approximately $1$-homogeneous.  Indeed, as $\theta \to 0$, the original rate-independent dynamics are recovered.

\begin{figure}[t]
	\begin{center}
	  \subfigure[A relatively ``hot'' thermalized gradient descent with $\theta = 1$.]
    {
			\psfrag{t}{{\small $t$}}
			\psfrag{x}[c]{{\small $y (t), z(t)$}}
			\includegraphics[width=0.45\textwidth]{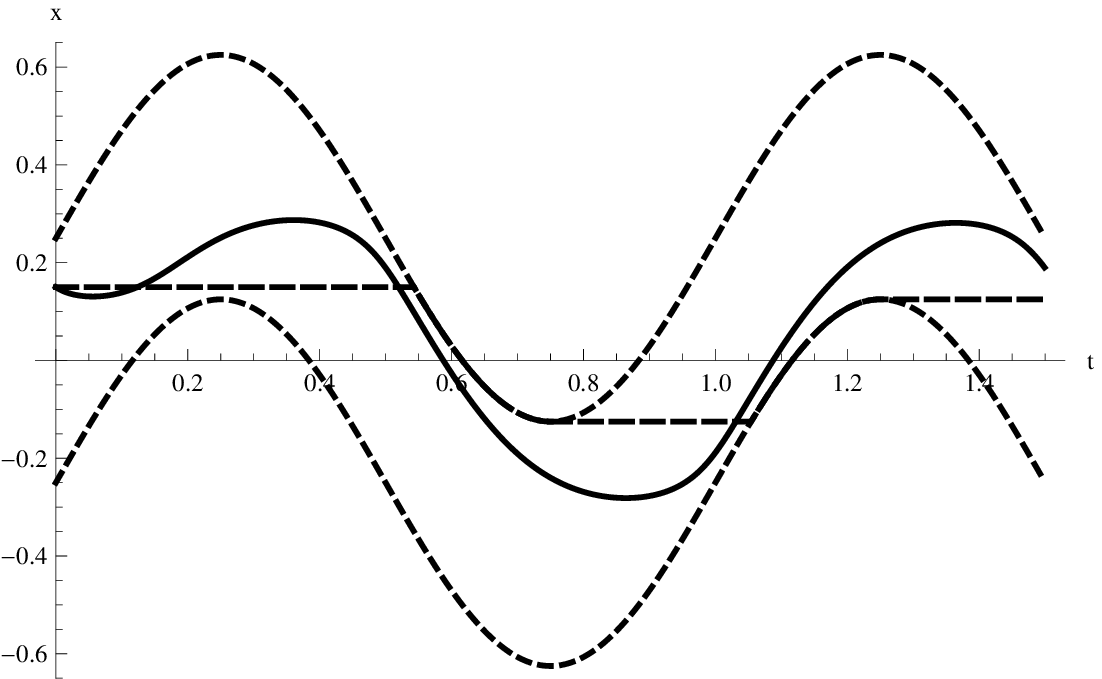}
    }
	  \subfigure[A relatively ``cold'' thermalized gradient descent with $\theta = 0.1$.]
    {
			\psfrag{t}{{\small $t$}}
			\psfrag{x}[c]{{\small $y (t), z(t)$}}
			\includegraphics[width=0.45\textwidth]{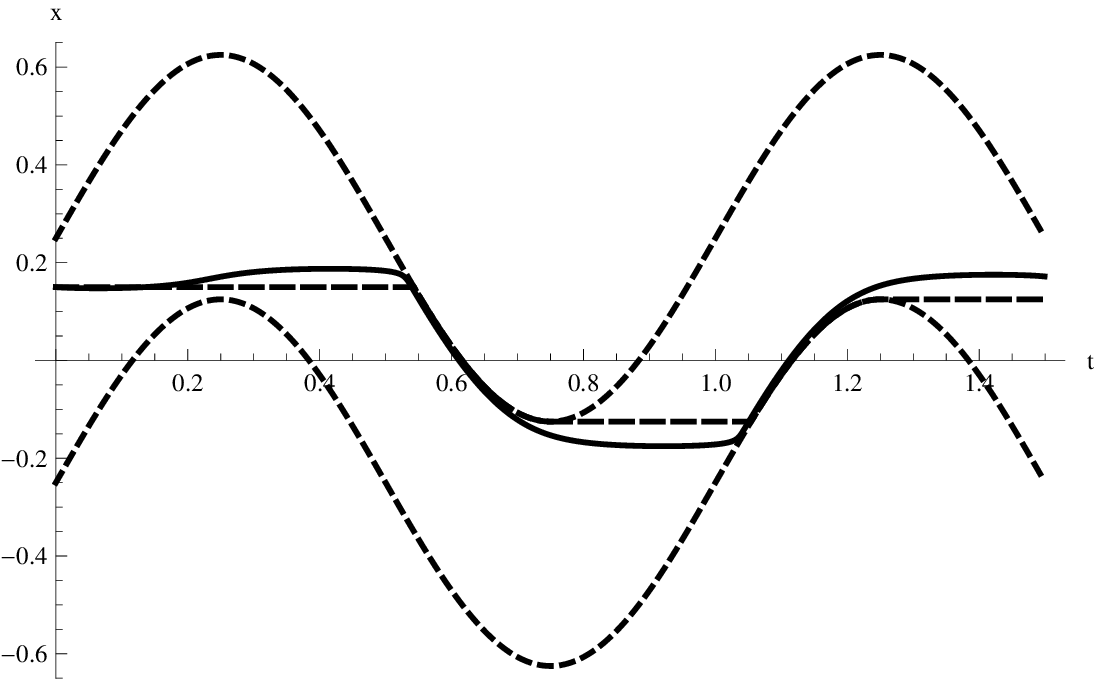}
    }
	  \caption[Comparison of the rate-independent process and the effect of the heat bath.]{Comparison of the original rate-independent evolution ($z$, dashed) and the effect of the heat bath ($y$, solid).  The frontier of the stable region is shown dotted.  Parameters:  $\Psi(x) = 2 | x |$, $E(t, x) = 4 | x |^{2} - \langle 3 \sin 2 \pi t, x \rangle$, initial condition $0.15$.}
	  \label{fig:RI-Relaxed}
	\end{center}
\end{figure}

\section{Conclusions and Outlook}
\label{sec:conclusions}

There are three natural directions in which the results of this paper could be generalized.  First and most obviously, the smoothness, convexity and other structural assumptions on $E$ could be relaxed:  so far, the various error terms in the proof of Theorem \ref{thm:main} have been controlled by convexity and conditions like \eqref{eq:NastyConvexityAssumption} and \eqref{eq:MarkovLim-Convex-USC};  in principle, so long as those error terms can be controlled (at least locally in time and space), Theorem \ref{thm:main} should generalize to the non-convex case.  This would be a very interesting generalization, since the solution to the rate-independent problem in a non-convex energetic potential $E$ is not always unique;  the thermalization procedure could provide a selection principle if the thermalized process has a unique limit as $\theta \to 0$.

Secondly, since most rate-independent processes of interest are infinite-\linebreak[4]dimensional, or even posed on spaces that lack a linear structure, more general state spaces than $\R^{n}$ could be considered.  This is a potentially subtle topic, since in infinite-dimensional settings there is no obvious candidate for a reference measure with respect to which to take densities to calculate the entropy in \eqref{eq:IP-Cost-2}.  As noted in \cite[Theorem 5.3.5]{Sullivan:2009}, the Markov chains of study are \emph{not} invariant under change of reference measure:  the logarithm of the Radon--Nikod{\'y}m derivative of the change of measure acts as an additive perturbation of the energetic potential.  The calculations of Section \ref{sec:heuristics} are quite interesting in general:  they amount to a study of the tangent measures (in the sense of \cite{Preiss:1987} \emph{\& al.}) of the Gibbsian distribution \eqref{eq:Gibbsian_pdf}.

Thirdly, the limiting result of Theorem \ref{thm:main} should be seen as a first-order approximation that is valid for small positive ``temperatures'' $\theta$.  It would be interesting to examine the behaviour of a suitable rescaling of $X - y$ and determine whether it obeys, say, the large deviations principle with respect to a suitable rate function.

\section{Proofs and Supporting Results}
\label{sec:proofs}

\begin{lemma}
	\label{lem:InfLemma}
	Let $\Psi \colon \R^{n} \to [0, +\infty)$ be one-homogeneous, continuous, and non-degenerate as in \eqref{eq:pos_homo_deg_1}--\eqref{eq:Psi_cC}.  Let $m \colon (\R^{n})^{\ast} \to \R$ be given by
	\[
		m(v) := \inf \left\{ \langle v, u \rangle + \Psi(u) \,\middle|\, u \in \mathbb{S}^{n - 1} \right\},
	\]
	where $\mathbb{S}^{n - 1} \subsetneq \R^{n}$ denotes the Euclidean unit sphere.  Then $m$ is continuous and
	\[
		m(v)
		\begin{cases}
			> 0, & \text{if $- v \in \mathring{\Elas}$,} \\
			= 0, & \text{if $- v \in \partial \Elas$,} \\
			< 0, & \text{if $- v \not\in \Elas$.}
		\end{cases}
	\]
\end{lemma}

\begin{proof}
	To save space, write $f(v, x) := \langle v, x \rangle + \Psi(x)$.  Since $\chi_{\Elas}$ is convex and lower semi-continuous,
	\begin{equation}
		\label{eq:InfLemma-eq1}
		\inf_{x \in \R^{n}} f(v, x) = - \Psi^{\star} (- v) = - \chi_{\Elas} (- v).
	\end{equation}
	Since $f(v, x)$ is $1$-homogeneous in $x$, it follows that $m(v) < 0$ for $- v \not\in \Elas$ and that $m(v) \geq 0$ for $- v \in \Elas$.  

	Note that $f$ is continuous.  Since $m$ is a pointwise infimum of a family of continuous functions, it is upper semi-continuous.  Since $\mathbb{S}^{n - 1}$ is compact, $m$ is a pointwise infimum of a compactly-parametrized family of continuous functions, and so is also lower semi-continuous \cite{PenotThera:1982}.  Thus, $m$ is continuous.

	Suppose that there exists $- v \in \mathring{\Elas}$ with $m(v) = 0$.  By the compactness of $\mathbb{S}^{n - 1}$, this implies that there exists a unit vector $u_{0} \in \mathbb{S}^{n - 1}$ with
	\begin{equation}
		\label{eq:InfLemma-eq2}
		f(v, u_{0}) = 0.
	\end{equation}
	But, since $- v \in \mathring{\Elas}$, there exists $\alpha > 1$ such that $- \alpha v \in \mathring{\Elas}$.  Then \eqref{eq:InfLemma-eq2} implies that $\langle v, u_{0} \rangle < 0$, so $f(\alpha v, u_{0}) < 0$, and so $m(\alpha v) < 0$, which contradicts \eqref{eq:InfLemma-eq1}.  Hence, $m(v) > 0$ for $- v \in \mathring{\Elas}$.

	It remains only to show that $m(v) = 0$ for $- v \in \partial \Elas$.  Suppose not, \emph{i.e.}\ that there exists $- v \in \partial \Elas$ with $m(v) > 0$.  Since $- v \in \partial \Elas$ and $\Elas$ is convex (and hence star-convex with respect to the origin in $(\R^{n})^{\ast}$), for every $\alpha > 1$, $- \alpha v \not \in \Elas$, and so $m(\alpha v) < 0$.  Hence, by the continuity of $m$,
	\[
		0 \geq \lim_{\alpha \searrow 1} m(\alpha v) = m(v) > 0,
	\]
	which is a contradiction.  This completes the proof.
\end{proof}

\begin{proof}[Proof of Proposition \ref{prop:EDDP}.]
	\label{pf:thm:EDDP}
	Let $\psi$ denote the Borel measure on $\R^{n}$ defined by \eqref{eq:psi_measure}.
	\begin{enumerate}
		\item By \eqref{eq:Psi_cC}, $\psi$ is a strictly positive and finite measure.  Hence, since the exponential function in the integrand of \eqref{eq:EDDP} is never zero, the claim follows.
	
		\item Consider the spherical integral form \eqref{eq:EDDP-Spherical} for $\edp^{\star}$. 	By Lemma \ref{lem:InfLemma}, if $- w \in \mathring{\Elas}$, then the integral is that of a continuous and bounded function over a compact set, so the integral exists and is finite.
	
		If $- w \in \partial \Elas$, then, as in the proof of Lemma \ref{lem:InfLemma}, there exists $u_{w} \in \mathbb{S}^{n - 1}$ with $\langle w, u_{w} \rangle + \Psi(u_{w}) = 0$, so the integrand has a pole.  The triangle inequality for $\Psi$ implies that for $u_{w} + u \in \R^{n}$,
		\begin{align*}
			\langle w, u_{w} + u \rangle + \Psi(u_{w} + u)
			& \leq \langle w, u_{w} \rangle + \langle w, u \rangle + \Psi(u_{w}) + \Psi(u) \\
			& = \langle w, u \rangle + \Psi(u) \\
			& \leq | u | \big( | w | + C_{\Psi} \big).
		\end{align*}
		Hence, the integrand in \eqref{eq:EDDP-Spherical} grows more quickly than $|u|^{-n}$ as $|u| \to 0$;  hence, by the standard result that $x \mapsto | x |^{- \alpha}$ lies in $L^{1}$ for a $d$-dimensional domain about $0$ if, and only if, $\alpha < d$, it follows that $\edp^{\star} (w) = + \infty$.
	
		If $- w \not \in \Elas$, then Lemma \ref{lem:InfLemma} implies that the integral in \eqref{eq:EDDP} does not converge, and so $\edp^{\star} (w) = + \infty$.

		\item Let $\mathcal{Z}_{\ep}^{\star}(w) := \exp \edp_{\ep}^{\star}(w)$, and let $v, w \in - \mathring{\Elas}$ and $0 < s < 1$.  Then
		\begin{align*}
			& \mathcal{Z}_{0}^{\star} ( s v + (1 - s) w ) \\
			& \quad = \int_{\R^{n}} \exp \big( - \langle s v + (1 - s) w, z \rangle \big) \, \rd \psi(z) \\
			& \quad = \int_{\R^{n}} \exp \big( - s \langle v, z \rangle \big) \exp \big( - (1 - s) \langle w, z \rangle \big) \, \rd \psi(z) \\
			& \quad \leq \left( \int_{\R^{n}} \exp \big( - \langle v, z \rangle \big) \, \rd \psi(z) \right)^{s} \left( \int_{\R^{n}} \exp \big( - \langle w, z \rangle \big) \, \rd \psi(z) \right)^{1 - s} \\
			& \quad = \mathcal{Z}_{0}^{\star} (v^{\ast})^{s} \mathcal{Z}_{0}^{\star} (w)^{1 - s},
		\end{align*}
		where the inequality follows from H{\"o}lder's inequality.  Hence, since the logarithm is a monotonically increasing function, for all $v, w \in - \mathring{\Elas}$,
		\[
			\edp^{\star} ( s v + (1 - s) w ) \leq s \edp^{\star} (v^{\ast}) + (1 - s) \edp^{\star} (w).
		\]
		Moreover, since $- \Elas$ is convex and $\edp^{\star}$ is identically $+ \infty$ outside the interior of $- \Elas$, $\edp^{\star}$ is convex on all of $(\R^{n})^{\ast}$.

		\item The derivative $\D \mathcal{Z}_{0}^{\star} \colon - \mathring{\Elas} \to (\R^{n})^{\ast \ast} \cong \R^{n}$ can be computed using the standard theorem on differentiation under the integral sign, yielding
		\[
			\D \mathcal{Z}_{0}^{\star} (v) = \int_{\R^{n}} - z \exp \big( - \langle v, z \rangle \big) \, \rd \psi(z),
		\]
		and so on for higher-order derivatives:
		\[
			\D^{k} \mathcal{Z}_{0}^{\star} (v) = \int_{\R^{n}} (- z)^{\otimes k} \exp \big( - \langle v, z \rangle \big) \, \rd \psi(z).	
		\]
		The integrals involved are all finite for $- v \in \mathring{\Elas}$ because of the exponentially small tails of the measure $\psi$.
	
		\item As in the proof of the second part of the claim, let $- w \in \mathring{\Elas}$ and let $u_{w} \in \mathbb{S}^{n - 1}$ be such that $\langle w, u_{w} \rangle + \Psi(u_{w})$ is minimal (\emph{i.e.}\ equals $m(w)$).  Then
		\[
			\langle w, u_{w} + u \rangle + \Psi(u_{w} + u) \leq m(w) + | u | \big( | w | + C_{\Psi} \big).
		\]
		Since $m(w) \to 0$ as $-w \to \partial \Elas$, the same argument as in point (2) applies, and so $\edp^{\star}(w) \to + \infty$ as $-w \to \partial \Elas$.  Now suppose that $\D \edp^{\star}$ does not blow up.  Then, since $\edp^{\star}$ is smooth and $\Elas$ is compact, $\edp^{\star}$ would be bounded on $-\Elas$, which is a contradiction. 	\qedhere
	\end{enumerate}
\end{proof}

\begin{proof}[Proof of Lemma \ref{lem:unif_cvx_implies_good}.]
	The energy evolution equation for $\edp^{\star}$ along $y^{0}$ can be calculated using the chain rule, yielding
	\begin{align*}
		& \frac{\rd}{\rd t} \edp^{\star} \big( \D E(t, y^{0}(t)) \big) \\
		& \quad = - \left\langle \D^{2} E(t, y^{0}(t)) , \D \edp^{\star} ( \D E(t, y^{0}(t)) )^{\otimes 2} \right\rangle \\
		& \quad \quad \quad + \left\langle \partial_{t} \D E(t, y^{0}(t)), \D \edp^{\star}(\D E(t, y^{0}(t))) \right\rangle \\
		& \quad \leq - \gamma_{E} \left| \D \edp^{\star} \big( \D E(t, y^{0}(t) \big) \right|^{2} + \| \partial_{t} \D E \|_{L^{\infty}} \left| \D \edp^{\star} \big( \D E(t, y^{0}(t)) \big) \right|
	\end{align*}
	Proposition \ref{prop:EDDP}(5) implies that if $\edp^{\star}$ blows up along any curve (\emph{i.e.}\ one that approaches $- \partial \Elas$ in the dual space), then so does $| \D \edp^{\star} |$.  However, the mean value theorem and the above calculation imply that $\edp^{\star}$ must be decreasing when $| \D \edp^{\star} |$ is large.  This yields the desired contradiction.
\end{proof}

The next lemma (Lemma \ref{lem:FepsToF0}) concerns the closeness of the effective dual dissipation potential $\edp^{\star}$ and the corresponding quantity $\edp_{\ep}^{\star}$ that controls the increments of the Markov chain.  Lemma \ref{lem:FepsToF0Flows} gives the resulting bound for the classical gradient descents in $\edp^{\star} \circ \D E$ and $\edp_{\ep}^{\star} \circ \D E$.  Both these two results apply to the prototypical case of a quadratic energetic potential.

\begin{lemma}
	\label{lem:FepsToF0}
	Suppose that the energetic potential $E$ is smooth enough that
	\[
		M := \sup_{t \in [0, T]} \sup_{k \geq 2} \big\| \D^{k} E(t, \cdot) \big\|_{\mathrm{op}} < + \infty,
	\]
	where $\| \cdot \|_{\mathrm{op}}$ denotes the operator norm.  Then, for every $K \Subset - \Elas$ and every $k \in \mathbb{N} \cup \{ 0 \}$, $\D^{k} \edp_{\ep}^{\star} \to \D^{k} \edp^{\star}$ uniformly on $K$ as $\ep \to 0$.  More precisely, for every such $K$ and $k$, there exists a constant $C \geq 0$ such that
	\[
		\sup_{w \in K} \big| \D^{k} \edp_{\ep}^{\star} (w) - \D^{k} \edp^{\star} (w) \big| \leq C \ep^{1/2} \text{ for all small enough } \ep > 0.
	\]
\end{lemma}

\begin{proof}
	The essential quantity to estimate is
	\[
		I_{k, \ep} (w) := \int_{\R^{n}} \left| 1 - \exp \left( - \sum_{\ell = 2}^{\infty} \frac{\ep^{\ell - 1}}{\ell !} \big\langle \D^{\ell} E(t, x), z^{\otimes \ell} \big\rangle \right) \right| | z |^{k} e^{- (\langle w, z \rangle + \Psi(z))} \, \rd z,
	\]
	since, by the elementary inequality
	\[
		\left| \frac{a}{b} - \frac{c}{d} \right| \leq \frac{| a - c |}{| d |} + \frac{| c | | b - d |}{| b d |},
	\]
	it holds true that
	\begin{equation}
		\label{eq:DkFepsErrorBound}
		| \D^{k} \edp_{\ep}^{\star} (w) - \D^{k} \edp^{\star} (w) | \leq \frac{1}{\mathcal{Z}_{0}^{\star} (w)} I_{k, \ep} (w) + \frac{| \D^{k} \mathcal{Z}_{0}^{\star} (w) |}{\mathcal{Z}_{0}^{\star} (w) \mathcal{Z}_{\ep} (w)} I_{0, \ep} (w).
	\end{equation}

	Let $m(w) := \inf \{ \langle w, z \rangle + \Psi(z) \mid | z | = 1 \}$.  By Lemma \ref{lem:InfLemma}, $m$ is continuous and bounded away from $0$ on $K$.  Similarly, since $\mathcal{Z}_{0}^{\star}$ and $\mathcal{Z}_{\ep}$ are continuous and positive, they are bounded away from $0$ on $K$, and $| \D^{k} \mathcal{Z}_{0}^{\star} |$ is bounded on $K$.  (Note that all these bounds fail on $- \partial \Elas$, so the assumption that $K \Subset - \Elas$ is essential.)  Thus, the emphasis is on estimating $I_{k, \ep} (w)$ in terms of $\ep$ and uniformly over $K$.  

	$I_{k, \ep} (w)$ will be estimated by splitting the integral into two parts:  an integral over a ball around the origin in $\R^{n}$, and an integral over the complement.  More precisely, for any $a \in (0, 1)$, let $R = R(a, \ep, x, t) > 0$ be such that
	\begin{equation}
		\label{eq:aandRBound}
		| z | \leq R \implies 1 - \exp \left( - \sum_{\ell = 2}^{\infty} \frac{\ep^{\ell - 1}}{\ell !} \langle \D^{\ell} E(t, x), z^{\otimes \ell} \rangle \right) \leq a.
	\end{equation}

	\noindent Converting to spherical polar coordinates yields that, for some constant $c_{n}$ depending only on $n$,
	\[
		I_{k, \ep} (w) \leq c_{n} a \int_{0}^{R} r^{k + n - 1} e^{- m(w) r} \, \rd r + c_{n} \int_{R}^{+ \infty} r^{k + n - 1} e^{- m(w) r} \, \rd r.
	\]
	This estimate is valid for any $a \in (0, 1)$ and corresponding $R$.  The above integrals can be evaluated exactly using the recurrence relation
	\[
		\int x^{n} e^{c x} \, \rd x = \frac{x^{n} e^{c x}}{c} - \frac{n}{c} \int x^{n - 1} e^{c x} \, \rd x;
	\]
	the resulting polynomial-exponential expressions are a bit cumbersome to deal with, but only the leading-order contributions as $\ep \to 0$ are of interest here.  

	We now make a specific choice of $a$ and $R$ such that $a \to 0$ and $R \to \infty$ at the right relative rates.  Let $a := \ep^{1/2}$ and let
	\[
		R := \frac{1}{\ep} \log \left( - \frac{\ep}{M} \log ( 1 - \ep^{1/2} ) \right).
	\]
	For any $z \in \R^{n}$,
	\begin{align*}
		\exp \left( - \sum_{\ell = 2}^{\infty} \frac{\ep^{\ell - 1}}{k !} \langle \D^{\ell} E(t, x), z^{\otimes \ell} \rangle \right)
		& = 
		\frac{1}{\ep} \sum_{\ell = 2}^{\infty} \frac{\ep^{\ell}}{\ell !} \big\langle \D^{\ell} E(t, x) , z^{\otimes \ell} \big\rangle \\
		& \leq \frac{1}{\ep} \sum_{\ell = 2}^{\infty} \frac{\ep^{\ell}}{\ell !} \big\| \D^{\ell} E(t, \cdot) \big\|_{\mathrm{op}} | z |^{\ell} \\
		& \leq \frac{M}{\ep} \exp ( \ep | z | ),
	\end{align*}
	and if $| z | \leq R$, then
	\[
		\frac{1}{\ep} \sum_{\ell = 2}^{\infty} \frac{\ep^{\ell}}{\ell !} \big\langle \D^{\ell} E(t, x) , z^{\otimes \ell} \big\rangle \leq - \log ( 1 - \ep^{1/2} ),
	\]
  as required for \eqref{eq:aandRBound} to hold.

	By l'H\^{o}pital's rule, for this choice of $a$ and $R$, $a \to 0$ and $R \to + \infty$ as $\ep \to 0$, and there exist constants $c_{1}$, $c_{2}$ such that
	\[
		I_{k, \ep} (w) \leq c_{1} \ep^{1/2} \left( \frac{R}{m(w)} \right)^{k + n - 1} e^{- m(w) R} + c_{2} \left( \frac{R}{m(w)} \right)^{k + n - 1} e^{- m(w) R}.
	\]

	\noindent The dominant term here is the $\ep^{1/2}$ term, since $R^{k + n - 1} e^{- m(w) R}$ not only tends to $0$, but does so with all derivatives tending to zero as well;  $m(w)$ is bounded away from zero for $w \in K \Subset - \Elas$.  Thus, there is a constant $C_{k}$ (dependent on $k$ and the other geometric parameters, but not on $\ep$) such that
	\[
	\sup_{w \in K} I_{k, \ep} (w) \leq C_{k} \ep^{1/2} \text{ for all small enough } \ep > 0.
	\]
	Thus, by \eqref{eq:DkFepsErrorBound}, as claimed
	\[
		\sup_{w \in K} | \D^{k} \edp_{\ep}^{\star} (w) - \D^{k} \edp^{\star} (w) | \leq C'_{k} \ep^{1/2} + C'_{0} \ep^{1/2} \in O(\ep^{1/2}) \text{ as } \ep \to 0. \qedhere
	\]
\end{proof}

\begin{lemma}
	\label{lem:FepsToF0Flows}
	Suppose that $y^{\ep}$, $\ep > 0$, and $y^{0}$ solve
	\begin{align*}
		\dot{y}^{\ep} & = - \D \edp_{\ep}^{\star} (\D E(t, y^{\ep})), \\
		\dot{y}^{0} & = - \D \edp^{\star} (\D E(t, y^{0})),
	\end{align*}
	with initial conditions $y^{\ep}(0) = y^{0}(0) = x_{0}$ such that $- \D E(0, x_{0}) \in \mathring{\Elas}$, and that \eqref{eq:MarkovLim-Convex-USC} holds.  Then there exists a constant $C \geq 0$ such that
	\[
		\sup_{t \in [0, T]} \big| y^{\ep}(t) - y^{0}(t) \big| \leq C \ep^{1/2} \text{ for all small enough } \ep > 0.
	\]
\end{lemma}

\begin{proof}
	The strategy is to appeal to Lemma \ref{lem:FepsToF0} and Gr{\"o}nwall's inequality. First, note that there exists a $K \Subset \Elas$ such that $- \D E(t, y^{0}(t)) \in K$ for all $t \geq 0$, \emph{i.e.}
	\[
		\inf_{t \in [0, T]} \mathrm{dist} \big( - \D E(t, y^{0}(t)), \partial \Elas \big) > 0,
	\]
	for otherwise, since $\edp^{\star}$ blows up to $+ \infty$ on $- \partial \Elas$ (by Proposition \ref{prop:EDDP}(5)), it would follow that $t \mapsto \edp^{\star} ( \D E(t, y^{0}(t)) )$ blows up to $+ \infty$ in finite time, which would contradict \eqref{eq:MarkovLim-Convex-USC}.

	Assume that $\ep > 0$ is small enough that the conclusion of Lemma \ref{lem:FepsToF0} holds.  Then, by Lemma \ref{lem:FepsToF0}, there exists $C \geq 0$ such that
	\[
		\sup_{t \in [0, T]} \big| \D \edp_{\ep}^{\star} (t, y^{0}(t)) - \D \edp^{\star} (t, y^{0}(t)) \big| \leq C \ep^{1/2}.
	\]
	Let $L$ be the product of the (finite) Lipschitz constants for $\D E$ and $\D \edp^{\star} |_{K}$.  Then, by Gr{\"o}nwall's inequality, for all $t \in [0, T]$,
	\[
		| y^{\ep}(t) - y^{0}(t) | \leq \frac{C}{L} \ep^{1/2}.  \qedhere
	\]
\end{proof}

\begin{lemma}
	\label{lem:ParabolicTightnessDiscrete}
	Consider a uniform partition $P$ of $[0, T]$ with $[P] = h > 0$.  Let $X$ be the Markov chain generated by $E$ (convex) and $\Psi$ (as usual) with $\ep = h$, and assume that \eqref{eq:NastyConvexityAssumption} holds.  Let $(y_{i})_{i = 1}^{N}$ be the Euler approximation to
	\[
		\dot{y} = - \D \edp_{h}^{\star} ( \D E(t, y(t)) )
	\]
	given by
	\[
		\Delta y_{i} := - h \D \edp_{h}^{\star} ( \D E(t_{i + 1}, y_{i})),
	\]
	with $X_{0} = y_{0}$ such that $- \D E(0, y_{0}) \in \mathring{\Elas}$.  Then, for every $\eta > 0$,
	\begin{equation}
		\label{eq:ParabolicTightnessDiscrete} %
		\mathbb{P} \left[ \max_{0 \leq i \leq T / h} \big| X_{i} - y_{i} \big| \geq \eta \right] \in O(h) \text{ as } h \to 0.
	\end{equation}
\end{lemma}

\begin{proof}
	In order to simplify the notation, assume that the partition $P$ is a uniform partition with $[P] = h > 0$, and define a time-dependent vector field $f_{h}$ by $f_{h} (t, x) := -\D \edp_{h}^{\star} (\D E(t, x))$.  Fix $\delta > 0$ small enough that
	\[
		K(t) := \{ x \in \Stab(t) \mid \mathrm{dist}(x, \partial \Stab(t)) > \delta \}
	\]
	is non-empty and contains $y(t)$ for every $t \in [0, T]$.  Furthermore, using the Lipschitz assumption on $\D E$, assume that $\delta$ is small enough that $K(t_{i}) \Subset \Stab(t_{i \pm 1})$ for each $i$.  Write (dropping the superscript that indicates the partition $P$ or its mesh size $h$)
	\begin{align*}
		X_{i + 1} &= X_{i} + h f_{h} (t_{i + 1}, X_{i}) + \Xi_{i + 1} (X_{i}), \\
		y_{i + 1} &= y_{i} + h f_{h} (t_{i + 1}, y_{i}).
	\end{align*}
	By Lemma \ref{lem:Markov-ExpAndMoments} (or, more precisely, its generalization to non-quadratic $E$ through \eqref{eq:Psitilde_eps}), for each $x$, $\Xi_{i + 1}(x)$ is a random variable with mean 0 and $k^{\mathrm{th}}$ central moment at most $C_{k}(x) h^{k}$.  The deviations $Z := X - y$ satisfy
	\begin{equation}
		\label{eq:RecurrenceForDeviations} %
		Z_{i + 1} = Z_{i} + h \big( f_{h} (t_{i + 1}, X_{i}) - f_{h} (t_{i + 1}, y_{i}) \big) + \Xi_{i + 1} (X_{i}).
	\end{equation}

	In summary, we have the following facts:
	\begin{itemize}
		\item[(M)] $f_{h} (t_{i + 1}, \cdot)$ is a monotonically decreasing vector field on $\Stab(t_{i + 1})$;
		\item[(B)] $f_{h} (t_{i + 1}, \cdot)$ is bounded on compactly-embedded subsets of $\Stab(t_{i + 1})$;
		\item[(Z)] for every $x$, $\E [ \Xi_{i + 1}(x) ] = 0$.
	\end{itemize}

	Let $\mathcal{K}_{i}$ be (the $\sigma$-algebra generated by) the event that $X_{j} \in K(t_{j + 1})$ for $0 \leq j \leq i$.  Applying the conditional expectation operator $\E [ \cdot | \mathcal{K}_{i} ]$ (which is never conditioning on an event of zero probability) to the Euclidean dot product of \eqref{eq:RecurrenceForDeviations} with itself yields that
	\begin{align*}
		& \E \big[ | Z_{i + 1} |^{2} \big| \mathcal{K}_{i} \big] - \E \big[ | Z_{i} |^{2} \big| \mathcal{K}_{i} \big] \\
		& \quad = 2 h \E \big[ (f_{h} (t_{i + 1}, X_{i}) - f_{h} (t_{i + 1}, y_{i})) \cdot Z_{i} \big| \mathcal{K}_{i} \big] && \text{$\leq 0$ by (M)} \\
		& \qquad + 2 \E \big[ Z_{i} \cdot \Xi_{i + 1} (X_{i}) \big| \mathcal{K}_{i} \big] && \text{$= 0$ by (Z)} \\
		& \qquad + 2 h \E \big[ (f_{h} (t_{i + 1}, X_{i}) - f_{h} (t_{i + 1}, y_{i})) \cdot \Xi_{i + 1} (X_{i}) \big| \mathcal{K}_{i} \big] && \text{$= 0$ by (Z)} \\
		& \qquad + h^{2} \E \big[ | f_{h} (t_{i + 1}, X_{i}) - f_{h} (t_{i + 1}, y_{i}) |^{2} \big| \mathcal{K}_{i} \big] && \text{$\leq C h^{2}$ by (B)} \\
		& \qquad + \E \big[ | \Xi_{i + 1} (X_{i}) |^{2} \big| \mathcal{K}_{i} \big] && \text{$\leq C h^{2}$ by Lemma \ref{lem:Markov-ExpAndMoments}} \\
		& \quad \leq C h^{2},
	\end{align*}
	and application of the unconditional expectation operator to both sides yields the following uniform bound for the second moment of the deviations:
	\begin{equation}
		\label{eq:Deviation2MomUnifBd} %
		\max_{0 \leq i \leq T / h} \E \big[ | Z_{i} |^{2} \big] \leq C T h.
	\end{equation}

	Inequality \eqref{eq:Deviation2MomUnifBd} can be used to ``bootstrap'' a similar inequality for the fourth moments. Define a tetralinear form $\tau \colon (\R^{n})^{4} \to \R$ by
	\begin{equation}
		\label{eq:TetralinearForm} %
		\tau(w, x, y, z) := (w \cdot x) (y \cdot z),
	\end{equation}
	so that $| x |^{4} = \tau(x, x, x, x)$. This tetralinear form is invariant under arbitrary compositions of the following interchanges of entries: $(1, 2)$, $(3, 4)$ and $(1, 3) (2, 4)$.  The Cauchy--Bunyakovski{\u\i}--Schwarz inequality for the Euclidean inner product implies a corresponding inequality for this tetralinear form: for all $w, x, y, z \in \R^{n}$,
	\begin{equation}
		\label{eq:CS4} %
		| \tau(w, x, y, z) | \leq | w | | x | | y | | z |.
	\end{equation}

	Hence, $\E \big[ | Z_{i + 1} |^{4} \big] \equiv \E \big[ \E \big[ | Z_{i + 1} |^{4} \big| \mathcal{K}_{i} \big] \big]$ can be expanded using the tetralinear form \eqref{eq:TetralinearForm} and \eqref{eq:RecurrenceForDeviations} and each term estimated as in the derivation of \eqref{eq:Deviation2MomUnifBd}.  By (Z), those terms containing precisely one $\Xi_{i + 1} (X_{i})$ have zero expectation; the terms of the form
	\[
		\E \left[ \left. \tau \big( Z_{i}, Z_{i}, Z_{i}, h (f_{h} (t_{i + 1}, X_{i}) - f_{h} (t_{i + 1}, y_{i})) \big) \right| \mathcal{K}_{i} \right]
	\]
	are non-positive by (M); the remaining terms can all be estimated using (B), \eqref{eq:Deviation2MomUnifBd}, \eqref{eq:CS4} and Lemma \ref{lem:Markov-ExpAndMoments}, with the worst bound being $O(h^{3})$. Thus, the following uniform bound for the fourth moment of the deviations holds:
	\begin{equation}
		\label{eq:Deviation4MomUnifBd} %
		\max_{0 \leq i \leq T / h} \E \big[ | Z_{i} |^{4} \big] \leq C T h^{2}.
	\end{equation}

	Hence, for $\eta > 0$,
	\begin{align*}
		& \mathbb{P} \big[ | Z_{i} | \geq \eta \text{ for some } 0 \leq i \leq T / h \big] \\
		& \quad \leq \sum_{i = 0}^{T / h} \mathbb{P} \big[ | Z_{i} | \geq \eta \big] \\
		& \quad \leq \sum_{i = 0}^{T / h} \eta^{-4} \E \big[ | Z_{i} |^{4} \big] && \text{by Chebyshev's inequality} \\
		& \quad \leq \eta^{-4} C T^{2} h && \text{by \eqref{eq:Deviation4MomUnifBd},}
	\end{align*}
	which establishes \eqref{eq:ParabolicTightnessDiscrete} and completes the proof.
\end{proof}

\bibliographystyle{amsplain}
\bibliography{./refs.bib}

\providecommand{\bysame}{\leavevmode\hbox to3em{\hrulefill}\thinspace}
\providecommand{\MR}{\relax\ifhmode\unskip\space\fi MR }
\providecommand{\MRhref}[2]{%
  \href{http://www.ams.org/mathscinet-getitem?mr=#1}{#2}
}
\providecommand{\href}[2]{#2}
\begin{thebibliography}{10}

\bibitem{AmbrosioGigliSavare:2008}
L.~Ambrosio, N.~Gigli, and G.~Savar{\'e}, \emph{Gradient {F}lows in {M}etric
  {S}paces and in the {S}pace of {P}robability {M}easures}, second ed.,
  Lectures in Mathematics ETH Z{\"u}rich, Birkh\"auser Verlag, Basel, 2008.
  \MR{2401600 (2009h:49002)}

\bibitem{BoydVandenberghe:2004}
S.~Boyd and L.~Vandenberghe, \emph{Convex {O}ptimization}, Cambridge University
  Press, Cambridge, 2004. \MR{2061575 (2005d:90002)}

\bibitem{JordanKinderlehrer:1996}
R.~Jordan and D.~Kinderlehrer, \emph{An extended variational principle},
  Partial {D}ifferential {E}quations and {A}pplications, Lecture Notes in Pure
  and Appl. Math., vol. 177, Dekker, New York, 1996, pp.~187--200. \MR{1371591
  (96m:49079)}

\bibitem{JordanKinderlehrerOtto:1998}
R.~Jordan, D.~Kinderlehrer, and F.~Otto, \emph{The variational formulation of
  the {F}okker--{P}lanck equation}, SIAM J. Math. Anal. \textbf{29} (1998),
  no.~1, 1--17 (electronic), \url{http://dx.doi.org/10.1137/S0036141096303359}.
  \MR{1617171 (2000b:35258)}

\bibitem{Koslowski:2003}
M.~Koslowski, \emph{A {P}hase-{F}ield {M}odel of {D}islocations in {D}uctile
  {S}ingle {C}rystals}, Ph.D. thesis, California Institute of Technology,
  Pasadena, California, USA, 2003.

\bibitem{Mielke:2005}
A.~Mielke, \emph{Evolution of rate-independent systems}, Evolutionary
  {E}quations.\ {V}ol.\ {II}, Handb. Differ. Equ., Elsevier/North-Holland,
  Amsterdam, 2005, pp.~461--559. \MR{2182832 (2007k:35498)}

\bibitem{Mielke:2007}
\bysame, \emph{Modeling and analysis of rate-independent processes}, January
  2007, Lipschitz Lecture held in Bonn:
  \url{http://www.wias-berlin.de/people/mielke/papers/Lipschitz07Mielke.pdf}.

\bibitem{MielkeRossiSavare:2009}
A.~Mielke, R.~Rossi, and G.~Savar{\'e}, \emph{Modeling solutions with jumps for
  rate-independent systems on metric spaces}, Discrete Contin. Dyn. Syst.
  \textbf{25} (2009), no.~2, 585--615,
  \url{http://dx.doi.org/10.3934/dcds.2009.25.585}. \MR{2525194 (2010k:49097)}

\bibitem{MielkeTheil:2004}
A.~Mielke and F.~Theil, \emph{On rate-independent hysteresis models}, NoDEA
  Nonlinear Differential Equations Appl. \textbf{11} (2004), no.~2, 151--189,
  \url{http://dx.doi.org/10.1007/s00030-003-1052-7}. \MR{2210284}

\bibitem{Moreau:1965}
J.-J. Moreau, \emph{Proximit{\'e} et dualit{\'e} dans un espace hilbertien},
  Bull. Soc. Math. France \textbf{93} (1965), 273--299. \MR{0201952 (34
  \#1829)}

\bibitem{PenotThera:1982}
J.-P. Penot and M.~Th{\'e}ra, \emph{Semicontinuous mappings in general
  topology}, Arch. Math. (Basel) \textbf{38} (1982), no.~2, 158--166,
  \url{http://dx.doi.org/10.1007/BF01304771}. \MR{650347 (83h:54011)}

\bibitem{Preiss:1987}
D.~Preiss, \emph{Geometry of measures in {$\mathbb{R}^{n}$}: distribution,
  rectifiability, and densities}, Ann. of Math. (2) \textbf{125} (1987), no.~3,
  537--643, \url{http://dx.doi.org/10.2307/1971410}. \MR{890162 (88d:28008)}

\bibitem{Sullivan:2009}
T.~J. Sullivan, \emph{Analysis of {G}radient {D}escents in {R}andom {E}nergies
  and {H}eat {B}aths}, Ph.D. thesis, Mathematitics Institute, University of
  Warwick, Coventry, UK, 2009.

\bibitem{SullivanKoslowskiTheilOrtiz:2009}
T.~J. Sullivan, M.~Koslowski, F.~Theil, and M.~Ortiz, \emph{On the behavior of
  dissipative systems in contact with a heat bath: application to {A}ndrade
  creep}, J. Mech. Phys. Solids \textbf{57} (2009), no.~7, 1058--1077,
  \url{http://dx.doi.org/10.1016/j.jmps.2009.03.006}. \MR{2535808
  (2010g:82060)}

\bibitem{Yosida:1965}
K.~Yosida, \emph{Functional {A}nalysis}, Die Grundlehren der Mathematischen
  Wissenschaften, Band 123, Academic Press Inc., New York, 1965. \MR{0180824
  (31 \#5054)}

\end{thebibliography}

\end{document}